\documentclass[10pt,a4paper]{article}
\usepackage[numbers,sort&compress]{natbib}
\usepackage[T1]{fontenc}
\usepackage[utf8]{inputenc}
\usepackage{authblk}
\usepackage{amsmath,amssymb,amsthm,esint,bm}
\usepackage{mathrsfs}
\usepackage{bookmark}
\usepackage{amsmath}
\usepackage{enumitem}
\allowdisplaybreaks[3]

\newtheorem{definition}{Definition}[section]
\newtheorem{theorem}[definition]{Theorem}
\newtheorem{lemma}[definition]{Lemma}

\newtheorem{corollary}[definition]{Corollary}
\theoremstyle{remark}
\newtheorem{remark}[definition]{Remark}
\numberwithin{equation}{section}

\newcommand{\abs}[1]{\lvert#1\rvert}
\newcommand{\Abs}[1]{\left\lvert#1\right\rvert}

\newcommand{\rn}{{\mathbb{R}^d}}

\allowdisplaybreaks[4]

\title{Higher H\"{o}lder regularity for
	 degenerate fully nonlinear elliptic equations with Hamiltonian terms}

\author[a]{Wentao Huo}
\author[a]{Xiaofeng Jin}
\author[b]{Lingwei Ma}
\author[b]{Zhenqiu Zhang\thanks{Corresponding author.}}
\affil[a]{School of Mathematical Sciences, Nankai University, Tianjin 300071, P.R. China}
\affil[b]{School of Mathematical Sciences and LPMC, Nankai University, Tianjin 300071, P.R. China}
\date{\today}

\usepackage{hyperref}
\begin{document}
	\maketitle
	\footnotetext[1]{E-mail: huowentaoouc@163.com (W. Huo), 1120220040@mail.nankai.edu.cn (X. Jin), malingwei@nankai.edu.cn (L. Ma), zqzhang@nankai.edu.cn (Z. Zhang).}

\begin{abstract}
This paper focuses on a class of fully nonlinear elliptic equations with variable double phase type degeneracy law and Hamiltonian terms. We obtain improved gradient H\"{o}lder regularity results  at points where the Hamiltonian coefficients and source terms vanish. Furthermore, we establish a Schauder-type estimate at local extrema, which is sharp with respect to the vanishing rate of the Hamiltonian coefficient and source term. Our approach adapts compactness and dichotomy arguments to capture the interplay between the degeneracy rate and the growth of the Hamiltonian term.		



\end{abstract}


\section{Introduction}\label{section1}
 Degenerate elliptic operators with Hamiltonian terms often arise in stochastic control as the Hamilton–Jacobi–Bellman equation satisfied by the value function of optimization problems governing degenerate diffusion processes\cite{Fleming,Leoni}. The present paper is devoted to the following fully nonlinear elliptic equations with nonhomogeneous double degeneracy and Hamiltonian terms
\begin{equation}\label{3aazhumodel}
\Phi(Du, x)F(D^2 u, x)+h(x)\abs{Du}^{m} =f(x) \quad  \text{in} \quad B_{1}:=B_{1}(0)\subset \mathbb{R}^{d},
\end{equation}
where $F$ is a uniformly $(\lambda,\Lambda)$-elliptic operator, $\Phi$ satisfies an
appropriate variable degeneracy law, and $f,h\in C(B_{1})\cap L^{\infty}(B_{1})$. We aim to investigate the quantitative effect of the H\"{o}lder-type vanish rate of the source term $f$ and the function $h$ on the regularity estimates. More precisely, let us state the structural assumptions on \eqref{3aazhumodel} that will be used throughout this work:
\begin{enumerate}[label=(\text{A}\arabic{enumi}),ref=\textbf{A}\arabic{enumi}]
	\item \label{A0} \textbf{(Uniform ellipticity)}
	The fully nonlinear operator
	$F:S^{d}\times B_{1}\rightarrow \mathbb{R}$
	is uniformly $(\lambda,\Lambda)$-elliptic in the sense that
	$$\lambda\|N\|\leq F(M+N,x)-F(M,x)\leq \Lambda\|N\|$$
	for some $0<\lambda\leq \Lambda$ and each $M,N\in S^{d}$ with $N\geq 0$. For normalization reasons, we shall assume $F(0,x)=0$ for all $x\in B_{1}$.
	\item \label{A1} \textbf{(H\"{o}lder continuity of the coefficients)}
	There exist constants $C>0$ and $\theta\in(0,1)$ such that
	\begin{equation}\label{lianxu}
		{\rm osc}_{F}(x,y):=\sup\limits_{M\in S^{d}\setminus \{0\}}\frac{\Abs{F(M,x)-F(M,y)}}{\|M\|}\leq C|x-y|^{\theta}
	\end{equation}
	for all $x,y\in B_{1}$. For convenience,
	we denote ${\rm osc}_{F}(x):={\rm osc}_{F}(x,0)$ and
	$$C_{F}:=\inf\left\{C>0:{\rm osc}_{F}(x,y)\leq C|x-y|^{\theta},\;\forall x,y\in B_{1}\right\}.$$
	\item \label{A2} \textbf{(Nonhomogeneous degeneracy)} The function $\Phi:\rn\times B_{1}\rightarrow [0,\infty)$ enjoys an appropriate variable degeneracy law, that is,
	\begin{equation}\label{3atuihuaxing}
		L_{1}	{\mathcal{K}}(|\xi|,x)\leq	{\Phi}(\xi,x) \leq L_{2} {\mathcal{K}}(|\xi|,x)
	\end{equation}
	for constants $0<L_{1}\leq L_{2}<\infty$,
	where
	\begin{equation}\label{3aaatuihuatiaojian2}
		{\mathcal{K}}(|\xi|,x):=|\xi|^{p(x)}+{a}(x)|\xi|^{q(x)},\quad (\xi,x)\in \rn\times B_{1}
	\end{equation}
	with functions $p(\cdot),q(\cdot)\in C(B_{1})$ and a modulating function $a(\cdot)$ fulfilling
	$$0\leq p_{min}\leq p(x)\leq q(x)\leq q_{max}<\infty\quad {\rm and}\quad 0\leq a(\cdot)\in C({B_{1}}).$$
	\item \label{A3} \textbf{(Assumptions on exponent $m$)} $0\leq m\leq 1+p_{min}$.
	\item \label{A4} \textbf{(Regularity of $h$)} The function $h$ is $C^{0,\beta}$ at 0 and $h(0)=0$. That is, we assume that for some $K_{0}>0$ and $\beta\in (0,1)$, there holds
	\begin{equation*}
		|h(x)|\leq K_{0}|x|^{\beta},\quad x\in B_{1}.
	\end{equation*}
    \item \label{A5} \textbf{(Regularity of $f$)} The source term $f$ is $C^{0,\gamma}$ at 0 and $f(0)=0$. That is, we assume that for some $K_{1}>0$ and  $\gamma\in (0,1)$, there holds
    \begin{equation*}
 	|f(x)|\leq K_{1}|x|^{\gamma},\quad x\in B_{1}.
    \end{equation*}
\end{enumerate}

Mathematical models like \eqref{3aazhumodel} are strongly inspired by certain variational integrals of the calculus of variations with nonstandard growth satisfying a sort of double-phase structure \cite{Marcellini1989,Mingione2015,Mingione201522,Zhikov1993,Zhikov1995}. One of the principal features of model \eqref{3aazhumodel} is its interplay between two distinct types of degeneracy rate, in accordance with the values of the modulating function.
For this reason, its diffusion process exhibits a kind of non-uniformly elliptic and doubly degenerate character. 

Over the past decades, equations of the type \eqref{3aazhumodel} have been studied quite extensively, including aspects such as  comparison principle, well-posedness of the Dirichlet problem, Aleksandroff–Bakelman–Pucci estimates, Harnack inequalities, and regularity theory (cf. \cite{Birindelli2004,Birindelli2006,Birindelli2007CPAA,Imbert2011JDE,Junges2010,Silva-Nornberg2021} and the
references therein). Among these, $C^{1,\alpha}$ regularity theory for solutions to \eqref{3aazhumodel} has been a central subject of research, playing a crucial role in the study of various free boundary problems of obstacle type \cite{SilvaDCDS,SilvaRM}, of one-phase Bernoulli type \cite{dasilva2023}, and of singular perturbation type \cite{Bezerra2023,Bezerra2025}.
In this direction, the seminal work \cite{Imbert1} showed that viscosity solutions to $|Du|^{p}F(D^{2}u)=f$ with $p\geq 0$ are locally $C^{1,\alpha}$ and the optimality of its H\"{o}lder exponent for the same problem is shown
in \cite{Ricarte}. Subsequently, this type of local regularity result has been extended to various kinds of degenerate fully nonlinear elliptic equations, we refer to \cite{Birindelli2014ESAIM,Birindelli2015,B-Demengel2016,B-Demengel2019,Bronzi2020,Fili,Silva2020,Fang,Baasandorj,Andrade,Silva2023} and the references therein.
In particular, \cite{Silva2023,Huo2026} established sharp interior $C^{1,\min\left\{\alpha_{0}^{-},\frac{1}{1+q_{max}}\right\}}$ regularity of \eqref{3aazhumodel} via a geometric tangential method and perturbation technique. Here $\alpha_{0}\in (0,1)$
is the H\"{o}lder exponent coming from the Krylov-Safonov regularity for homogeneous equation $F(D^{2}u)=0$ (see \cite[Chapter 5]{Caff1}).

At this point, a natural question arises: can this regularity be improved, at least at some meaningful points? A first observation  comes from  an important example in \cite{Imbert1}. For any $\alpha\in (0,1)$, the function $u(x)=|x|^{1+\alpha}$ satisfies
$$|Du|^{p}\Delta u=(1+\alpha)^{1+p}(d+\alpha-1)|x|^{(1+\alpha)(1+p)-(2+p)}.$$
Letting $f(x)=C|x|^{\gamma}$ with $0<\gamma:=(1+\alpha)(1+p)-(2+p)$, one can see that,
$u$ is $C^{1,\frac{1+\gamma}{1+p}}$ at the origin (a critical point of $u$). This suggests that the known gradient H\"{o}lder regularity of $u$ can be surpassed, at least at some meaningful point. In a very recent paper \cite{Nascimento2025}, Nascimento considered the following fully nonlinear equations with single degeneracy law
\begin{equation*}
	\abs{Du}^{p}F(D^2 u, x) =f(x) \quad  \text{in} \quad B_{1},
\end{equation*}
where $p>0$, $f$ is $C^{0,\gamma}$ at 0 and $f(0)=0$. By means of a geometric approach, it was shown that solution $u$ is of class $C^{1,\min\left\{\alpha_{0}^{-},\frac{1+\gamma}{1+p}\right\}}$ at the origin. Regarding earlier similar results for $p$-degenerate elliptic equations of divergence type $-{\rm div}\left(\mathcal{A}(x,Du)\right)=f$, please refer to \cite{Teixeira2014MathAnn,Teixeira2022ARMA}. Nonetheless, to the best of our knowledge, no such results are yet available concerning fully nonlinear elliptic equations simultaneously involving degeneracy and Hamiltonian terms.

Another heuristic analysis comes from a scaling argument. To be precise, consider a simpler equation
\begin{equation*}
	|Du|^{p}\Delta u+h(x)|Du|^{m}=f(x) \quad {\rm in}\;\, B_{1}.
\end{equation*}
For any $r\in(0,1)$, define $v(x):=u(rx)/r^{\vartheta}$ for $x\in B_{1}$. It is clear that $v$ solves
\begin{equation*}
	|Dv|^{p}\Delta v+\tilde{h}(x)|Dv|^{m}=\tilde{f}(x) \quad {\rm in}\;\, B_{1},
\end{equation*}
where
\begin{equation*}
	\tilde{h}(x)=r^{2+p-m-\vartheta(1+p-m)}h(rx)\quad {\rm and}\quad \tilde{f}(x)=r^{2+p-\vartheta(1+p)}f(rx).
\end{equation*}
Thus, if we select
\begin{equation*}
	0<\vartheta\leq \min\left\{\frac{2+p+\gamma}{1+p},\frac{2+p-m+\beta}{1+p-m}\right\},
\end{equation*}
then it follows from $|h(x)|\leq K_{0} |x|^{\beta}$, $|f(x)|\leq K_{1} |x|^{\gamma}$ and $r\in (0,1)$ that
\begin{equation*}
	\abs{\tilde{h}(x)}\leq K_{0} r^{2+p-m-\vartheta(1+p-m)+\beta}|x|^{\beta}\leq K_{0} |x|^{\beta},
\end{equation*}
\begin{equation*}
	\abs{\tilde{f}(x)}\leq K_{1} r^{2+p-\vartheta(1+p)+\gamma}|x|^{\gamma}\leq K_{1} |x|^{\gamma},
\end{equation*}
that is, $\tilde{f}$ and $\tilde{h}$ satisfy the same structural assumptions as $f$ and $h$.
Therefore, $v$ satisfies an equation with the same structure as $u$. Given $0<m\leq 1+p$, we find
\begin{equation*}
	 \min\left\{\frac{2+p+\gamma}{1+p},\frac{2+p-m+\beta}{1+p-m}\right\}=1+\min\left\{\frac{1+\gamma}{1+p},\frac{1+\beta}{1+p-m}\right\}=:1+\varsigma
\end{equation*}
Hence, the optimal regularity can be expected to be $C^{1,\varsigma}$. In particular, if $p<\min\left\{\gamma,\beta+m\right\}$, then it follows $\varsigma>1$ and so the optimal regularity can be expected to be $C^{2,\varsigma-1}$.

Motivated by the analysis and work mentioned above, the primary objective of this paper is to establish higher regularity estimates for viscosity solutions to \eqref{3aazhumodel} in a unified way.  It should be emphasized that the key challenge stems from the simultaneous presence of the nonhomogeneous degeneracy law and the Hamiltonian term. To address this, we provide a refined analysis to capture the competition effect between the degeneracy rate and the growth of the Hamiltonian term.

In what follows, a constant is said to be universal if it depends only on the structure constants appearing in assumptions \eqref{A0}-\eqref{A5}. Our first result is an improved gradient H\"{o}lder regularity at the origin point, at which the source term $f$ and function $h$ vanish at a prescribed rate.
\begin{theorem}\label{main}
	 Assume that  hypotheses \eqref{A0}-\eqref{A5} hold.
	Let $u\in C(B_{1})$ be a bounded viscosity solution of \eqref{3aazhumodel}. Then $u$ is of class $u\in C^{1,\min\left\{\alpha_{0}^{-},\alpha_{1}\right\}}$ at the origin with
	\begin{equation*}
		\alpha_{1}:=\min\left\{\frac{1+\gamma}{1+q_{max}},\frac{1+\beta}{1+q_{max}-m}\right\}
		.
	\end{equation*}
That is, given any
\begin{equation}\label{3aerfadefanwei}
	\alpha\in
	(0,\alpha_{0})\cap \left(0,\alpha_{1}\right],
\end{equation}
there exists universal constants
$0<r<\frac{1}{4}$ and $C>0$ such that
\begin{equation*}
	\Abs{u(x)-u(0)-Du(0)\cdot x}\leq C|x|^{1+\alpha}\quad {\rm for\; all \;} x\in B_{r}.
\end{equation*}
\end{theorem}
\begin{remark}\label{coro1}
	 It should be noted that, in contrast to the sharp  $C^{1,\min\left\{\alpha_{0}^{-},\frac{1}{1+q_{max}}\right\}}$ regularity for solutions with bounded source term and bounded function $h$ (see \cite{Bezerra2023,Huo2026}),
	Theorem \ref{main} exhibits a significant improvement in smoothness. In fact, when $\gamma,\beta>0$, it holds
	$$\min\left\{\frac{1+\gamma}{1+q_{max}},\frac{1+\beta}{1+q_{max}-m}\right\}>\frac{1}{1+q_{max}}.$$
	Moreover, it is noteworthy that
	$$\min\left\{\frac{1+\gamma}{1+q_{max}},\frac{1+\beta}{1+q_{max}-m}\right\}\geq 1,$$
	provided that $q_{max}\leq \min\left\{\gamma,m+\beta\right\}$. Therefore, the solution $u$ to \eqref{3aazhumodel} is of class $ C^{1,\alpha_{0}^{-}}$ at the origin, in other words, the solution of \eqref{3aazhumodel} is asymptotically as regular as $F$-harmonic function.
\end{remark}

Since viscosity solutions to convex/concave equations $F(D^{2}u)=0$ are locally of class $C^{1,1}$ by classical  Evans–Krylov theory \cite{Evans, Krylov1}. This result together with Theorem \ref{main} yields the following optimal regularity result.
\begin{corollary}
Suppose that the assumptions of Theorem \ref{main}  hold. Suppose further that operator $F$ is convex (or concave) and $q_{max}\leq \min\left\{\gamma,m+\beta\right\}$. Then $u$ is of class $ C^{1,1^{-}}$ at the origin. That is, given any
	$\alpha\in (0,1)$, there exists universal constants $0<r<\frac{1}{4}$ and $C>0$ such that
	\begin{equation*}
		\Abs{u(x)-u(0)-Du(0)\cdot x}\leq C|x|^{1+\alpha}\quad {\rm for\; all \;} x\in B_{r}.
	\end{equation*}
\end{corollary}
The final result provided in this work is a Hessian continuity at a local extrema.
\begin{theorem}\label{3aamain2}
	Suppose  that hypotheses \eqref{A0}-\eqref{A5} hold.
	Let $u\in C(B_{1})$ be a bounded viscosity solution of \eqref{3aazhumodel}. Assume further $q_{max}<\min\left\{\gamma,m+\beta\right\}$, and that origin is a local extrema of $u$, i.e., $u(0)\leq u(x)$ or $u(0)\geq u(x)$ in $B_{\nu}(0)$ for some $\nu\in (0,1/4)$. 	Then $u$ is twice differentiable at the origin and
	\begin{equation}\label{3ajielun2}
		|u(x)-u(0)|\leq C|x|^{2+\mu} \quad {\rm for\; all \;} x\in B_{r},
	\end{equation}
	where
	\begin{equation*}
		\mu:=\min\left\{\frac{\gamma-q_{max}}{1+q_{max}},\frac{m+\beta-q_{max}}{1+q_{max}-m}\right\}
	\end{equation*}
	and universal constants $0<r<\nu$ and $C>0$. That is, $u$ is
	of class $u\in C^{2,\mu}$ at the origin with $|Du(0)|=|D^{2}u(0)|=0$.
\end{theorem}
\begin{remark}
	Owing to the generality of the degeneracy term and the presence of the Hamiltonian term, our findings above encompass and extend the regularity results previously obtained in \cite{Nascimento2025}. 
However, due to the abstract form of $p(x)$ and $q(x)$ and the simultaneous presence of $h(x)|Du|^{m}$, implementing the strategy from \cite{Nascimento2025} becomes a delicate task. In addition, \eqref{3aazhumodel} is no longer homogeneous caused by the presence of different and variable gradient power, which makes the scaling process more tricky. In contrast with the single power-type degeneracy law, the quantities involving the gradient variable on the left hand side of \eqref{3aazhumodel} are not identically preserved after scaling.
\end{remark}
\begin{remark}
An extension of our results also hold to multi-phase  degenerate fully nonlinear equation with Hamiltonian terms modelled by
\begin{equation*}
	\bigg(|Du|^{p(x)}+\sum_{i=1}^{N}a_{i}(x)|Du|^{q_{i}(x)}\bigg)F(D^2 u, x)+h(x)|Du|^{m}
	=f(x) \quad  \text{in} \quad B_{1},
\end{equation*}
where $0\leq a_{i}(\cdot)\in C(B_{1})$, $p(\cdot),q_{i}(\cdot)\in C(B_{1})$, $i\in\left\{1,2,\cdots,N\right\}$, and $0\leq p_{min}\leq p(x)\leq q_{1}(x)\leq q_{2}(x)\leq\cdots\leq q_{N}(x)\leq q_{max}<\infty$. 
\end{remark}

The remainder of this paper is organized as follows. In Section \ref{section2}, we present some basic notions and explain the scaling property. 
In Section \ref{section3}, we establish the improved gradient regularity result. Finally, in last section, we complete the proof of Theorem \ref{3aamain2} regarding Schauder-type regularity estimate.
\section{Preliminaries}\label{section2}
\subsection{Notations and basic concepts}
Throughout this paper, let $S^{d}$ be the set of all real symmetric $d\times d$ matrices and $B_{r}(x_{0})$ be the open ball with radius $r$ and centred at $x_{0}\in\rn$. In particular, we shall simply denote $B_{r}:=B_{r}(0)$. In addition, $C$ denotes a constant whose value may vary from line to line, and only the relevant dependencies are specified in parentheses.

Let us first introduce the notion of viscosity
solution for the following equations
\begin{equation}\label{model22}
	G(D^{2}u,Du,x):=f(x)-\Phi(Du, x)F(D^2 u, x)-h(x)\abs{Du}^{m}=0 \quad  \text{in} \quad B_{1}.
\end{equation}
\begin{definition}
	A function $u\in C(B_{1})$ is a viscosity supersolution (resp. subsolution) to \eqref{model22}, if whenever $\varphi \in C^2\left(B_{1}\right)$ and $x_0 \in B_{1}$ such that $u -\varphi$ has a local minimum (resp. local maximum) at $x_0$, then
	$$
	G(D^2\varphi(x_0),D\varphi(x_0),x_0) \geq  0 \quad({\rm resp.}\; G(D^2\varphi(x_0),D\varphi(x_0),x_0) \leq  0).
	$$
	Finally, a function $u$ is said to be a viscosity solution of \eqref{model22} if it is simultaneously a viscosity supersolution and a viscosity subsolution.
\end{definition}

To simplify our presentation let us introduce the following definitions.
\begin{definition}
	Given a function $v\in C^{1}(B_{1})$, we denote the set of zero critical points as
	$$\mathscr{C}(v):=\{x\in B_{1}:v(x)=|Dv(x)|=0\}.$$
\end{definition}
\begin{definition}
	Let $F$ be as in \eqref{A0}.
	A function $v\in C(B_{1})$ is said to be
	$F$-harmonic in $B_{1}$ if it is a viscosity solution of $F(D^{2}v)=0$ in $B_{1}$.
\end{definition}

\subsection{Scaling feature}\label{scaling}
In this subsection, we start off by making several comments on the scaling features of the model \eqref{3aazhumodel}, which will be used frequently throughout the paper. Let $u$ be a viscosity solution to \eqref{3aazhumodel}. For constants $K\geq 1\geq \tau>0$ arbitrary, define $\tilde{u}:B_{1}\rightarrow \mathbb{R}$ as
$$\tilde{u}(x):=\frac{u(\tau x)}{K}.$$
We can readily check that
$\tilde{u}$ is a viscosity solution of
\begin{equation} \label{3shensuofangcheng}
	\tilde{\Phi}(D\tilde{u}, x)	\tilde{F}(D^2 \tilde{u}, x)+	\tilde{h}(x)\Abs{D\tilde{u}}^{m}= 	\tilde{f}(x) \quad \text{in} \quad  B_{1},
\end{equation}
where
\begin{align*}
	\tilde{F}(X,x):=&\frac{\tau^{2}}{K}F\left(\frac{K}{\tau^{2}}X,\tau x\right),\\
	\tilde{\Phi}(\xi,x):=&\left(\frac{\tau}{K}\right)^{p(\tau x)}\Phi\left(\frac{K}{\tau}\xi,\tau x\right),\\
	\tilde{\mathcal{K}}(|\xi|,x):=&|\xi|^{p(\tau x)}+\tilde{a}(x)|\xi|^{q(\tau x)},\\
	\tilde{a}(x):=&\left(\frac{\tau}{K}\right)^{p(\tau x)-q(\tau x)}a(\tau x),\\
	\tilde{h}(x):=&\frac{\tau^{2+p(\tau x)-m}}{K^{1+p(\tau x)-m}}h(\tau x),\\
	\tilde{f}(x):=&\frac{\tau^{2+p(\tau x)}}{K^{1+p(\tau x)}}f(\tau x).
\end{align*}
Note that $\tilde{F}$ is still a uniformly $(\lambda,\Lambda)$-elliptic operator. A direct calculation yields
\begin{equation*}
	\begin{split}
		{\rm osc}_{\tilde{F}}(x,0)
		&=\sup\limits_{M\in S^{d}\setminus \{0\}}\frac{\Abs{F\left(\frac{K}{\tau^{2}}M,\tau x\right)-F\left(\frac{K}{\tau^{2}}M,0\right)}}{\frac{K}{\tau^{2}}\|M\|}={\rm osc}_{{F}}(\tau x,0).
	\end{split}
\end{equation*}
Moreover,
\begin{equation*}
	L_{1}	\tilde{\mathcal{K}}(|\xi|,x)\leq	\tilde{\Phi}(\xi,x) \leq L_{2}	\tilde{\mathcal{K}}(|\xi|,x)\quad {\rm for}\;\,(\xi,x)\in\mathbb{R}^{d}\times B_{1},
\end{equation*}
that is, $\tilde{\Phi}$ satisfies the same structural assumption as $\Phi$. It follows from \eqref{A4}, \eqref{A5}, $K\geq 1\geq \tau>0$, and $m\leq 1+p_{min}$ that
\begin{equation*}
	\abs{\tilde{f}(x)}\leq \frac{\tau^{2+p_{\rm min}}}{K^{1+p_{\rm min}}}\abs{f(\tau x)}\leq K_{1}\tau^{2+p_{\rm min}+\gamma}|x|^{\gamma}\leq K_{1}|x|^{\gamma},
\end{equation*}
\begin{equation*}
	\abs{\tilde{h}(x)}\leq \frac{\tau^{2+p_{\rm min}-m}}{K^{1+p_{\rm min}-m}}\abs{h(\tau x)}\leq K_{0}\tau^{2+p_{\rm min}-m+\beta}|x|^{\beta}\leq K_{0}|x|^{\beta},
\end{equation*}
which means that $\tilde{f}$ and $\tilde{h}$ satisfy the same structural assumptions as $f$ and $h$. Hence,
$\tilde{u}$ satisfies equation \eqref{3shensuofangcheng}, which is subject to the same structural assumptions as the equation for $u$.

In particular, up to a normalization, i.e., by choosing $K:=1+\|u\|_{L^{\infty}(B_{1})}$, we can assume, with no loss of generality, that $\|u\|_{L^{\infty}(B_{1})}\leq 1$ is a normalized solution of \eqref{3aazhumodel}. In addition, since $u(x)-u(0)$ satisfies the same equation as $u(x)$, by translation, we can assume $u(0)=0$.
\section{Improved graident regularity}\label{section3}
This section is dedicated to the proof of Theorem \ref{main} concerning an improved H\"{o}lder gradient regularity at the origin. We begin by proving a $F$-harmonic approximation result, which essentially states that, under suitable smallness assumptions, a normalized viscosity solution $u$
of \eqref{3aazhumodel} in $B_1$ can be approximated by a $F$-harmonic function $v$ with $0\in \mathscr{C}(v)$, provided that $Du(0)$ is near zero.
\begin{lemma}
	\label{lem3.1}
	Let the conditions \eqref{A0}-\eqref{A3} be in force. Let $u\in C(B_{1})$ be a normalized viscosity solution of \eqref{3aazhumodel} with
	 $u(0)=0$. Given $\varepsilon>0$, there exists constant $\delta>0$ depending on $d,\lambda,\Lambda,\varepsilon$ such that if
	\begin{equation*}
	\max\left\{|Du(0)|,\|{\rm osc}_{{F}}\|_{L^{\infty}\left(B_{1}\right)},\|f\|_{L^{\infty}(B_{1})},\|h\|_{L^{\infty}(B_{1})}\right\}\leq \delta,
	\end{equation*}
	 then one can find a $F$-harmonic function $v$ with $v\in C_{loc}^{1,\alpha_{0}}(B_{1})$, so that $0\in \mathscr{C}(v)$ and
	 \begin{equation*}
	 	\|u-v\|_{L^{\infty}(B_{1/2})}\leq \varepsilon.
	 \end{equation*}
\end{lemma}
\begin{proof}
	Argue by contradiction. If the claim fails, then there exist $\varepsilon_{0}>0$ and sequences
	of functions $\{F_{j}\}_{j\in \mathbb{N}}$, $\{\Phi_{j}\}_{j\in \mathbb{N}}$, $\{f_{j}\}_{j\in \mathbb{N}}$, $\{u_{j}\}_{j\in \mathbb{N}}$, $\{h_{j}\}_{j\in \mathbb{N}}$ such that
	\begin{equation}\label{model333}
		\Phi_{j}(Du_{j}, x)F_{j}(D^2 u_{j}, x)+h_{j}(x)\abs{Du_{j}}^{m} =f_{j}(x) \quad  \text{in} \quad B_{1}
	\end{equation}	
with $ \|u_{j}\|_{L^{\infty}(B_{1})}\leq 1$ and $u_{j}(0)=0$, as well as
\begin{equation}\label{3aajintiaojiana}
	\max\left\{|Du_{j}(0)|,\|{\rm osc}_{{F_{j}}}\|_{L^{\infty}\left(B_{1}\right)},\|f_{j}\|_{L^{\infty}(B_{1})},\|h_{j}\|_{L^{\infty}(B_{1})}\right\}\leq \frac{1}{j},
\end{equation}
where $F_{j}:S^{d}\times B_{1}\rightarrow \mathbb{R}$ is uniformly $(\lambda,\Lambda)$-elliptic, $\Phi_{j}$ satisfies \eqref{3atuihuaxing} and \eqref{3aaatuihuatiaojian2} with
$$0\leq p_{min}\leq p_{j}(x)\leq q_{j}(x)\leq q_{max}<\infty,$$
$$p_{j}(\cdot),q_{j}(\cdot)\in C(B_{1}) \quad {\rm and}\quad 0\leq a_{j}(\cdot)\in C(B_{1}).$$
	Nonetheless, it holds
	 \begin{equation}\label{3bijinmaodun}
		\|u_{j}-v\|_{L^{\infty}(B_{1/2})}>\varepsilon_{0}\quad {\rm for\; any \;} j\in \mathbb{N},
	\end{equation}
	for all $F$-harmonic function $v$ with $0\in \mathscr{C}(v)$.
	
	By the uniform ellipticity of $F_{j}$ and \eqref{3aajintiaojiana}, we can see that $F_{j}$ converge locally uniformly to some uniformly $(\lambda,\Lambda)$-elliptic
	operator $F_{\infty}$ (with frozen coefficients). In addition,  we know from \cite[Theorem 1.1]{Huo2026} that the sequence $\{u_{j}\}_{j\in\mathbb{N}}\subset C_{loc}^{1,\alpha^{\prime}}(B_{1})$ with $\alpha^{\prime}\in(0,1)$. Therefore, by applying Arzel${\rm \grave{a}}$-Ascoli theorem, we conclude that, up to a subsequence, $u_{j}$ converges locally uniformly to some continuous function $u_{\infty}$ in $B_{1}$ in the $C^{1}$-topology.
	Now, by arguing as \cite[Lemma 4.1]{Huo2026}, we can conclude that
	 $u_{\infty}$ is a viscosity solution of
	 \begin{equation}\label{homojie}
	 	F_{\infty}(D^{2}u_{\infty})=0 \quad {\rm in}\quad  B_{3/4}
	 \end{equation}
with $u_{\infty}(0)=|D u_{\infty}(0)|=0$. Finally, taking $v=u_{\infty}$, we reach a contradiction with \eqref{3bijinmaodun} for $j$ sufficiently large. This completes the proof of the desired result.
\end{proof}

\begin{lemma}\label{lem3.2}
	 Suppose that the hypotheses of Lemma \ref{lem3.1} are in force. Given $\alpha\in(0,\alpha_{0})$, there exist $\delta>0$ and $0<\rho<\frac{1}{2}$ depending on $d,\lambda,\Lambda,\alpha$, such that if
	\begin{equation*}
	\max\left\{|Du(0)|,\|{\rm osc}_{{F}}\|_{L^{\infty}\left(B_{1}\right)},\|f\|_{L^{\infty}(B_{1})},\|h\|_{L^{\infty}(B_{1})}\right\}\leq \delta,
\end{equation*}
then there holds
	\begin{equation*}
	\sup\limits_{x\in B_{\rho}}|u(x)|\leq \rho^{1+\alpha}.
	\end{equation*}
\end{lemma}
\begin{proof}
	Let $\varepsilon>0$ be a number to be chosen later. According to Lemma \ref{lem3.1}, we can find a $\delta>0$ and
	a $F$-harmonic function $v$ such that $0\in \mathscr{C}(v)$ and
	\begin{equation}\label{33bijintiaohe}
		\sup\limits_{x\in B_{1/2}}|u(x)-v(x)|\leq \varepsilon.
	\end{equation}
	Applying the optimal $C_{loc}^{1,\alpha_{0}}$ regularity for $v$ (see e.g., \cite{Caffarelli1989,Teixeira2014}), together with the fact $v(0)=|D v(0)|=0$, we derive
\begin{equation*}
	\sup\limits_{x\in B_{\rho}}|v(x)|\leq C\rho^{1+\alpha_{0}}\quad {\rm for\;any}\;\rho\in\left(0,\frac{1}{2}\right)
\end{equation*}
with a universal constant $C=C(d,\lambda,\Lambda)>0$.
This along with \eqref{33bijintiaohe} yields that
\begin{equation}\label{33zhengzexingdiyibu}
	\sup\limits_{x\in B_{\rho}}\Abs{u(x)}\leq \sup\limits_{x\in B_{\rho}}\Abs{u(x)-v(x)}+\sup\limits_{x\in B_{\rho}}\Abs{v(x)}\leq \varepsilon+C\rho^{1+\alpha_{0}}.
\end{equation}	
Now, fixed an exponent $\alpha<\alpha_{0}$, we choose $\rho$ and $\varepsilon$ as
\begin{equation*}
	0<\rho\leq \left(\frac{1}{2C}\right)^{1/(\alpha_{0}-\alpha)}\quad {\rm and}\quad \varepsilon:=\frac{1}{2}\rho^{1+\alpha}.
\end{equation*}
Then it follows from \eqref{33zhengzexingdiyibu} that
\begin{equation*}
	\sup\limits_{x\in B_{\rho}}\Abs{u(x)}\leq \rho^{1+\alpha}.
\end{equation*}	
This completes the proof of the desired result.	
\end{proof}
\begin{lemma}\label{lem3.3}
		 Suppose that hypotheses \eqref{A0}-\eqref{A5} are in force.
		 Let  $u\in C(B_{1})$ be a normalized viscosity solution of \eqref{3aazhumodel} with $u(0)=0$ and $\alpha$ be as in \eqref{3aerfadefanwei}.
	 There exist constants $\delta>0$ and $0<\rho<\frac{1}{2}$ both of which are the same as
	 those in Lemma \ref{lem3.2}, such that if
		 \begin{equation*}
		 |D u(0)|\leq \delta t^{\alpha}\quad {\rm for}\;t\in (0,\rho],
		 \end{equation*}
		 then it holds that
	\begin{equation*}
		\sup\limits_{x\in B_{t}}\Abs{u(x)}\leq Ct^{1+\alpha}
	\end{equation*}
for a universal constant $C>0$.
\end{lemma}
\begin{proof}
	We verify this claim by induction argument.
	Observe that, by making use of the scaling features presented in subsection \ref{scaling}, we may suppose that
	\begin{equation}\label{smalljiashe}
		\max\left\{\|{\rm osc}_{{F}}\|_{L^{\infty}\left(B_{1}\right)},\|f\|_{L^{\infty}(B_{1})},\|h\|_{L^{\infty}(B_{1})}\right\}\leq \delta
	\end{equation}
for $\delta$ coming from Lemma \ref{lem3.1}. In fact, we define rescaled function $\tilde{u}:B_{1}\rightarrow \mathbb{R}$ by $\tilde{u}(x)=\frac{u(\tau x)}{K}$
with
\begin{equation*}
	K:=1+\|u\|_{L^{\infty}\left(B_{1}\right)}
\end{equation*}
and
\begin{equation*}
	\tau:=
		 \min\left\{1,\left(\frac{\delta}{C_{F}}\right)^{\frac{1}{\theta}},\left(\frac{\delta}{\|f\|_{L^{\infty}(B_{1})}}\right)^{\frac{1}{2+p_{min}}},\left(\frac{\delta}{\|h\|_{L^{\infty}(B_{1})}}\right)^{\frac{1}{2+p_{min}-m}}\right\}.
\end{equation*}
Then $\tilde{u}$ satisfies \eqref{3shensuofangcheng} with $\tilde{u}(0)=0$ and $\|\tilde{u}\|_{L^{\infty}\left(B_{1}\right)}\leq 1$.
In addition, a direct calculation yields that	
\begin{equation*}
	\|{\rm osc}_{\tilde{F}}\|_{L^{\infty}\left(B_{1}\right)}= \|{\rm osc}_{{F}}\|_{L^{\infty}\left(B_{\tau}\right)}\leq C_{F}\tau^{\theta},
\end{equation*}
\begin{equation*}
	\|\tilde{f}\|_{L^{\infty}\left(B_{1}\right)}\leq \frac{\tau^{2+p_{min}}}{K^{1+p_{min}}}\|{f}\|_{L^{\infty}\left(B_{1}\right)},
\quad
	\|\tilde{h}\|_{L^{\infty}\left(B_{1}\right)}\leq \frac{\tau^{2-m+p_{min}}}{K^{1-m+p_{min}}}\|{h}\|_{L^{\infty}\left(B_{1}\right)}.
\end{equation*}
Then it follows from the choice of $K$ and $r$ above that
\begin{equation*}
	\max\left\{\|{\rm osc}_{\tilde{F}}\|_{L^{\infty}\left(B_{1}\right)},\|\tilde{f}\|_{L^{\infty}\left(B_{1}\right)},\|\tilde{h}\|_{L^{\infty}\left(B_{1}\right)}\right\}
	\leq \delta.
\end{equation*}		
	
First of all, we need to justify that for
all $k\in\mathbb{N}$, it holds that	if
	\begin{equation*}
	|D u(0)| \leq \delta \rho^{k\alpha},
\end{equation*}
then
	\begin{equation}\label{3111aaa}
		\sup\limits_{x\in B_{\rho^{k}}}\Abs{u(x)}\leq \rho^{k(1+\alpha)}.
	\end{equation}
We argue by finite induction. For $k=1$, \eqref{3111aaa} follows immediately from Lemma \ref{lem3.2}. Suppose that \eqref{3111aaa} holds true for $j=1,2,...,k$. Now we are going to prove \eqref{3111aaa} for $j=k+1$. To this end, we define  $u_{k}:B_{1}\rightarrow \mathbb{R}$ by
	\begin{equation*}
	u_{k}(x):=\frac{u\left(\rho^{k}x\right)}{\rho^{k(1+\alpha)}}.
\end{equation*}
We can readily check that $u_{k}$ solves
\begin{equation*}
	\Phi_{k}(D{u_{k}}, x){F}_{k}(D^2 {u_{k}}, x)+{h_{k}}(x)\Abs{D{u_{k}}}^{m}= {f_{k}}(x) \quad \text{in} \quad  B_{1}
\end{equation*}
in the viscosity sense, where
\begin{align*}
	{F_{k}}(X,x):=&\rho^{k(1-\alpha)}F\left(\rho^{k(\alpha-1)}X,\rho^{k}x\right),\\
	{\Phi_{k}}(\xi,x):=&\rho^{-k\alpha p_{k}(x)}\Phi\left(\rho^{k\alpha}\xi,\rho^{k}x\right)\;{\rm satisfies}\;\eqref{3atuihuaxing}\;{\rm with}\\
	p_{k}(x)&:=p(\rho^{k}x),\quad q_{k}(x):=q(\rho^{k}x), \quad
	a_{k}(x):=\rho^{k\alpha(q_{k}(x)-p_{k}(x))}a(\rho^{k}x),\\
	{h_{k}}(x):=&\rho^{k(1-\alpha(1+p_{k}(x)-m))}h(\rho^{k}x)
	,\quad
	{f_{k}}(x):=\rho^{k(1-\alpha(1+p_{k}(x)))}f(\rho^{k}x).
\end{align*}
It is clear that $F_{k}$ is also a uniformly $(\lambda,\Lambda)$-elliptic operator and
\begin{equation*}
	\|{\rm osc}_{{F_{k}}}\|_{L^{\infty}\left(B_{1}\right)}=\|{\rm osc}_{{F}}\|_{L^{\infty}\left(B_{\rho^{k}}\right)}\leq \|{\rm osc}_{{F}}\|_{L^{\infty}\left(B_{1}\right)}.
\end{equation*}
Moreover, $0\leq p_{min}\leq p_{k}(x)\leq q_{k}(x)\leq q_{max}<\infty.$
Applying $|h(x)|\leq K_{0}|x|^{\beta}$ and $|f(x)|\leq K_{1}|x|^{\gamma}$ for $x\in B_{1}$, in combination with $\alpha\leq  \min\left\{\frac{1+\gamma}{1+q_{max}},\frac{1+\beta}{1+q_{max}-m}\right\}$, we deduce that
\begin{equation}\label{331}
	|{f_{k}}(x)|\leq K_{1} \rho^{k(1-\alpha(1+p_{k}(x))+\gamma)}|x|^{\gamma}
	\leq K_{1} \rho^{k(1-\alpha(1+q_{max})+\gamma)}|x|^{\gamma}\leq K_{1}|x|^{\gamma},
\end{equation}
\begin{equation}\label{332}
	|{h_{k}}(x)|\leq  K_{0} \rho^{k(1-\alpha(1+p_{k}(x)-m)+\beta)}|x|^{\beta}
	\leq K_{0} \rho^{k(1-\alpha(1+q_{max}-m)+\beta)}|x|^{\beta}\leq K_{0}|x|^{\beta}.
\end{equation}
In addition, via the hypotheses of induction, we know $|D u(0)| \leq \delta \rho^{k\alpha}$, and so we
get $$|D u_{k}(0)|=\rho^{-k\alpha}|Du(0)| \leq \delta.$$ At
this point, in view of the assumption \eqref{smalljiashe}, $u_{k}$ falls into the framework of Lemma \ref{lem3.2}, and hence it yields
\begin{equation*}
	\sup\limits_{x\in B_{\rho}}\Abs{u_{k}(x)}\leq \rho^{1+\alpha}.
\end{equation*}
Scaling back, we reach that
\begin{equation*}
	\sup\limits_{x\in B_{\rho^{k+1}}}\Abs{u(x)}\leq \rho^{(k+1)(1+\alpha)}.
\end{equation*}
By now, the proof of \eqref{3111aaa} is finished.

Finally, for $t\in(0,\rho]$, there exists an integer $j\in \mathbb{N}$ such that $\rho^{j+1}<t\leq \rho^{j}$,
applying  \eqref{3111aaa} to obtain
\begin{equation*}
	\sup\limits_{x\in B_{t}}\Abs{u(x)}\leq
	\sup\limits_{x\in B_{\rho^{j}}}\Abs{u(x)}\leq \rho^{j(1+\alpha)}\leq \rho^{-(1+\alpha)}t^{1+\alpha},
\end{equation*}
provided $|Du(0)|\leq \delta t^{\alpha}$.
This completes the proof of the desired result.
\end{proof}

In the end, with the help of Lemma \ref{lem3.3}, we are in a position to complete the proof of Theorem \ref{main}.
\begin{proof}[Proof of Theorem \ref{main}] As previously explained in Subsection \ref{scaling}, we can assume that $u(0)=0$, and $\|{u}\|_{L^{\infty}\left(B_{1}\right)}\leq 1$
by translation and normalization. We are going to
prove this conclusion by implementing a dichotomy argument, which is divided into two cases. Let $\delta,\rho>0$ be two universal constants coming from Lemma \ref{lem3.3}.

{\bf Case 1.} $|Du(0)|\leq \delta \rho^{\alpha}$. Set
	\begin{equation}\label{31112changshujifa}
		\kappa:=\left(\frac{|Du(0)|}{\delta}\right)^{1/\alpha}.
	\end{equation}
Given $0<t\leq \rho$. We consider the following two subcases.\\
{\bf Case 1.1} $\kappa\leq t\leq \rho$. It follows from \eqref{31112changshujifa} that
\begin{equation*}
	|Du(0)|=\delta \kappa^{\alpha} \leq \delta t^{\alpha}.
\end{equation*}
Consequently, exploiting Lemma \ref{lem3.3} to obtain
\begin{equation*}
	\sup\limits_{x\in B_{t}}\Abs{u(x)}\leq Ct^{1+\alpha}
\end{equation*}
for a universal constant $C>0$. Combining the last two displays with the triangle inequality, we derive
\begin{equation*}
	\sup\limits_{x\in B_{t}}\Abs{u(x)-Du(0)\cdot x}\leq
	\sup\limits_{x\in B_{t}}\Abs{u(x)}+|Du(0)|t\leq (C+\delta)t^{1+\alpha},
\end{equation*}
which means that $u$ is of class $C^{1,\alpha}$ at 0.\\
{\bf Case 1.2} $0<t<\kappa\leq \rho$. In this case, we consider the scaled function
\begin{equation*}
	u_{\kappa}(x):=\frac{u(\kappa x)}{\kappa^{1+\alpha}}, \quad x\in B_{1}.
\end{equation*}
It follows from \eqref{31112changshujifa} that
$|Du(0)|=\delta \kappa^{\alpha}$. Again, we employ Lemma \ref{lem3.3} to derive
\begin{equation*}
	\sup\limits_{x\in B_{\kappa}}\Abs{u(x)}\leq C\kappa^{1+\alpha}.
\end{equation*}
Then it follows that
\begin{equation*}
	\sup\limits_{x\in B_{1}}\Abs{u_{\kappa}(x)}=\sup\limits_{x\in B_{1}}\Abs{\frac{u(\kappa x)}{\kappa^{1+\alpha}}}=\frac{1}{\kappa^{1+\alpha}}\sup\limits_{x\in B_{\kappa}}\Abs{u(x)}
	\leq C.
\end{equation*}
Therefore, we can readily examine that $u_{\kappa}\in C(B_{1})$ is a bounded viscosity solution to
\begin{equation} \label{switchequation}
	{\Phi}_{\kappa}(D{u}_{\kappa}, x)	{F}_{\kappa}(D^2 {u}_{\kappa}, x)+	{h}_{\kappa}(x)\Abs{D{u}_{\kappa}}^{m}= 	{f}_{\kappa}(x) \quad \text{in} \quad  B_{1},
\end{equation}
where ${F}_{\kappa}$ has the same uniform ellipticity constants as the operator $F$,
\begin{align*}
		{{\Phi}_{\kappa}}(\xi,x):=&\kappa^{-\alpha {p}_{\kappa}(x)}\Phi\left(\kappa^{\alpha}\xi,\kappa x\right)\;{\rm satisfies}\;\eqref{3atuihuaxing}\;{\rm with}\\
	{p}_{\kappa}(x)&:=p(\kappa x),\quad {q}_{\kappa}(x):=q(\kappa x), \quad
	{a}_{\kappa}(x):=\kappa^{\alpha({q}_{\kappa}(x)-p_{\kappa}(x))}a(\kappa x),\\
	{h}_{\kappa}(x):=&\kappa^{1-\alpha(1+p_{\kappa}(x)-m)}h(\kappa x),\quad
	{f}_{\kappa}(x):=\kappa^{1-\alpha(1+p_{\kappa}(x))}f(\kappa x).
\end{align*}	
By the same arguments as in \eqref{331} and \eqref{332}, we can deduce that
\begin{equation*}
	|{h}_{\kappa}(x)|\leq K_{0}|x|^{\beta},\quad  |{f}_{\kappa}(x)|\leq K_{1}|x|^{\gamma}.
\end{equation*}
Since the source term $f_{\kappa}$ and function $h_{\kappa}$ have a universal bound, the interior
$C^{1,\alpha^{\prime}}$-regularity for $u_{\kappa}$ follows from \cite[Theorem 1.1]{Huo2026}. In addition, let us recall the fact that
$|Du_{\kappa}(0)|=\Abs{\frac{Du(0)}{\kappa^{\alpha}}}=\delta>0.$ Since $u_{\kappa}\in C_{loc}^{1,\alpha^{\prime}}(B_{1})$, there exists a universal constant $C> 0$ such that for all $x\in B_{1/2}$,
$$\Abs{Du_{\kappa}(x)-Du_{\kappa}(0)}\leq C|x|^{\alpha^{\prime}},$$
and further
$$ \delta-C|x|^{\alpha^{\prime}}\leq \Abs{Du_{\kappa}(x)}\leq \delta+C|x|^{\alpha^{\prime}}.$$
Then we may take a small universal radius $r>0$, independent of $\kappa$, such that
$$\frac{\delta}{2}\leq |Du_{\kappa}(x)| \leq 2\delta \quad {\rm in}\;\, B_{r}.$$
Hence, $u_{\kappa}$ satisfies the following uniformly elliptic equation with a universally bounded source term on the right-hand side
\begin{equation*}
		F_{\kappa}(D^2 u_{\kappa}, x)=\Phi_{\kappa}(Du_{\kappa}, x)^{-1}\bigg[f_{\kappa}(x)- h_{\kappa}(x)\Abs{Du_{\kappa}}^{m} \bigg] \quad \text{in} \quad  B_{r}.
\end{equation*}
As a consequence, according to the regularity results established in \cite[Section 4]{Teixeira2014}, solution to the above equation is almost as regular as a $F$-harmonic function,
i.e., $u_{\kappa}\in C_{loc}^{1,\alpha_{0}^{-}}(B_{r})$. In particular, we obtain
\begin{equation*}
	\sup\limits_{x\in B_{\iota}}\Abs{u_{\kappa}(x)-Du_{\kappa}(0)\cdot x}\leq C\iota^{1+\alpha}
\end{equation*}
for every $0<\iota\leq \frac{r}{2}$ and $0<\alpha<\alpha_{0}$. Scaling back, it yields that
\begin{equation}\label{3zuoquyujielunchengli}
	\sup\limits_{x\in B_{t}}\Abs{u(x)-Du(0)\cdot x}\leq Ct^{1+\alpha}
\end{equation}
for every $0<t\leq \frac{\kappa r}{2}$. It remains to show that the claim \eqref{3zuoquyujielunchengli} also holds on interval $\left(\frac{\kappa r}{2},\kappa\right)$. When $t\in\left(\frac{\kappa r}{2},\kappa\right)$, we apply Case 1.1 with $t=\kappa$ to arrive at
\begin{equation*}
	\begin{split}
			\sup\limits_{x\in B_{t}}\Abs{u(x)-Du(0)\cdot x}&\leq \sup\limits_{x\in B_{\kappa}}\Abs{u(x)-Du(0)\cdot x}\\
			& \leq C\kappa^{1+\alpha}=C \left(\frac{2}{r}\right)^{1+\alpha}\left(\frac{\kappa r}{2}\right)^{1+\alpha}\\
			&\leq C \left(\frac{2}{r}\right)^{1+\alpha}t^{1+\alpha}.
	\end{split}
\end{equation*}
At this point, we complete the proof of Case 1.

{\bf Case 2.} $|Du(0)|>\delta \rho^{\alpha}$. In this scenario, we consider an auxiliary function
$$\hat{u}(x):=\frac{\delta \rho^{\alpha}}{|Du(0)|}u(x)$$
and then $|D\hat{u}(0)|=\delta \rho^{\alpha}$, which is back to Case 1. The proof is complete now.
\end{proof}
\section{Schauder-type regularity}
\label{section4}
In order to give the proof of Theorem \ref{3aamain2}, we start off with a new flatness improvement result at the local extrema.
\begin{lemma}\label{3aalem4.1}
Let the conditions \eqref{A0}-\eqref{A3} be in force. Let  $u\in C(B_{1})$ be a normalized viscosity solution of \eqref{3aazhumodel}. Assume that $x_{0}\in B_{1/2}$ is a local extremum, i.e., $u(x_{0})\leq u(x)$ or $u(x_{0})\geq u(x)$ in $B_{\nu}(x_{0})$ for some $\nu\in(0,1/4)$.  Given $\varepsilon>0$, there exists $\delta=\delta(d,\lambda,\Lambda,\varepsilon)>0$ such that if
\begin{equation*}
	\max\left\{\|{\rm osc}_{{F}}\|_{L^{\infty}\left(B_{1}\right)},\|f\|_{L^{\infty}(B_{1})},\|h\|_{L^{\infty}(B_{1})}\right\}\leq\delta,
\end{equation*}
then it holds
\begin{equation*}
	\sup\limits_{x\in B_{\nu}(x_{0})}\Abs{u(x)-u(x_0)}\leq \varepsilon.
\end{equation*}
\end{lemma}
\begin{proof}
	 With no loss of generality, let us assume $u(x_{0})$ is a local minimum. The proof is based on a contradiction argument. If the claim fails, then there exist $\varepsilon_{0}>0$ and sequences
	of functions $\{F_{j}\}_{j\in \mathbb{N}}$, $\{\Phi_{j}\}_{j\in \mathbb{N}}$, $\{f_{j}\}_{j\in \mathbb{N}}$, $\{u_{j}\}_{j\in \mathbb{N}}$, $\{h_{j}\}_{j\in \mathbb{N}}$, and a sequence $\{x_{j}\}\subset B_{1/2}$ of local minimum of $u_{j}$ such that
	\begin{equation*}
		\Phi_{j}({Du_{j}}, x)F_{j}(D^2 u_{j}, x)+h_{j}(x)\abs{Du_{j}}^{m} =f_{j}(x) \quad  \text{in} \quad B_{1}
	\end{equation*}	
	with $ \|u_{j}\|_{L^{\infty}(B_{1})}\leq 1$ and
	\begin{equation*}
		\max\left\{\|{\rm osc}_{{F_{j}}}\|_{L^{\infty}\left(B_{1}\right)},\|f_{j}\|_{L^{\infty}(B_{1})},\|h_{j}\|_{L^{\infty}(B_{1})}\right\}\leq \frac{1}{j},
	\end{equation*}
	where operator $F_{j}:S^{d}\times B_{1}\rightarrow \mathbb{R}$ is uniformly $(\lambda,\Lambda)$-elliptic, $\Phi_{j}$ satisfies \eqref{3atuihuaxing} and \eqref{3aaatuihuatiaojian2} with
	$$0\leq p_{min}\leq p_{j}(x)\leq q_{j}(x)\leq q_{max}<\infty,$$
	$$p_{j}(\cdot),q_{j}(\cdot)\in C(B_{1}) \quad {\rm and}\quad 0\leq a_{j}(\cdot)\in C(B_{1}).$$
	However, it holds
	\begin{equation}\label{42maodun}
		\sup\limits_{x\in B_{\nu}(x_{j})}\Abs{u_{j}(x)-u_{j}(x_j)}>\varepsilon_{0}\quad {\rm for\; any \;} j\in \mathbb{N}.
	\end{equation}
	
By the same arguments as in Lemma \ref{lem3.1}, we can conclude that
	$u_{\infty}$ is a normalized viscosity solution of
	\begin{equation*}
		F_{\infty}(D^{2}u_{\infty})=0 \quad {\rm in}\quad  B_{3/4}.
	\end{equation*}
Additionally, $x_{j}\rightarrow x_{\infty}$. Since $u_{j}$ converges locally uniformly in $B_{1}$ to $u_{\infty}$,  we know that $x_{\infty}$ is a
local minimum of $u_{\infty}$. Thus, by the strong maximum principle for uniformly elliptic
operators (see \cite[Proposition 4.9]{Caff1}), we have $u_{\infty}\equiv const$. This clearly leads to a contradiction with \eqref{42maodun} for $j$ sufficiently large. 
\end{proof}
With the aid of Lemma \ref{3aalem4.1}, we now show higher regularity of solution to \eqref{3aazhumodel}.
\begin{proof}[Proof of Theorem \ref{3aamain2}]
 As argued before, we can assume that $u(0)=0$, and $\|{u}\|_{L^{\infty}\left(B_{1}\right)}\leq 1$
by translation and normalization. Let us also assume 0 is a local minimum of $u$. For constant $0<\varrho<1$ to be chosen later, consider the scaled function
$$w_{1}(x):=u(\varrho x),\quad x\in B_{1}.$$
Then we immediately verify that $w_{1}$ solves
\begin{equation*}
	\Phi_{\varrho}({D{w_{1}}}, x){F}_{\varrho}(D^2 {w_{1}}, x)+{h_{\varrho}}(x)\Abs{D{w_{1}}}^{m}= {f_{\varrho}}(x) \quad \text{in} \quad  B_{1}
\end{equation*}
in the viscosity sense, where
\begin{align*}
	{F_{\varrho}}(X,x):=&\varrho^{2}F\left(\varrho^{-2}X,\varrho x\right),\\
	{\Phi_{\varrho}}(\xi,x):=&\varrho^{ p_{\varrho}( x)}\Phi\left(\varrho^{-1}\xi,\varrho x\right)\;{\rm satisfies}\;\eqref{3atuihuaxing}\;{\rm with}\\
		p_{\varrho}(x):&=p(\varrho x),\quad q_{\varrho}(x):=q(\varrho x),\quad
	a_{\varrho}(x):=\varrho^{p_{\varrho}(x)-q_{\varrho}(x)}a(\varrho x),\\
	{h_{\varrho}}(x):=&\varrho^{2+p_{\varrho}( x)-m}h(\varrho x),\quad
	{f_{\varrho}}(x):=\varrho^{2+p_{\varrho}( x)}f(\varrho x).
\end{align*}		
Through a straightforward calculation, we get
\begin{equation*}
	\|{\rm osc}_{{F_{\varrho}}}\|_{L^{\infty}\left(B_{1}\right)}=\|{\rm osc}_{{F}}\|_{L^{\infty}\left(B_{\varrho}\right)}\leq C_{F}\varrho^{\theta}.
\end{equation*}
Moreover, $0\leq p_{min}\leq p_{\varrho}(x)\leq q_{\varrho}(x)\leq q_{max}<\infty.$ It follows from \eqref{A4}-\eqref{A5} that
\begin{align*}
	|{f_{\varrho}}(x)|&\leq K_{1} \varrho^{2+p_{min}+\gamma}|x|^{\gamma}\leq K_{1} \varrho^{2+p_{min}+\gamma}, \\
	|{h_{\varrho}}(x)|&\leq K_{0} \varrho^{2+p_{min}-m+\beta}|x|^{\beta}\leq K_{0} \varrho^{2+p_{min}-m+\beta}.
\end{align*}
Now, we fix $\varepsilon:=\nu^{2+\mu}$ and let $\delta>0$ be the corresponding smallness condition from Lemma \ref{3aalem4.1}. We proceed by selecting $\varrho$
\begin{equation}\label{canshuxuanqu}
0<\varrho\leq \min\left\{\left(\frac{\delta}{C_{F}}\right)^{\theta},\left(\frac{\delta}{K_{0}}\right)^{1/(2+p_{min}-m+\beta)},\left(\frac{\delta}{K_{1}}\right)^{1/(2+p_{min}+\gamma)}\right\}.
\end{equation}
With such choice, $w_{1}$ is under the hypotheses of Lemma \ref{3aalem4.1}, and hence it yields
\begin{equation}\label{41}
	\sup\limits_{x\in B_{\nu}}\Abs{w_{1}(x)}\leq \nu^{2+\mu}.
\end{equation}

As usual, we argue by induction to show for all $k\in \mathbb{N}$, it holds that
\begin{equation}\label{disanhanshu}
	\sup\limits_{x\in B_{\nu^{k}}}\Abs{w_{1}(x)}\leq \nu^{k(2+\mu)}.
\end{equation}
The initial induction hypothesis, $k=1$, is precisely  \eqref{41}. Assume \eqref{disanhanshu} has been verified for $j=k$. Now we prove that \eqref{disanhanshu} holds for $j=k+1$. Let us define $$w_{k}(x):=\nu^{-k(2+\mu)}w_{1}(\nu^{k}x),\quad x\in B_{1}.$$
It is clear that 0 is still a local minimum for $w_{k}$. By induction assumption, we have $$\|w_{k}\|_{L^{\infty}(B_{1})}\leq 1.$$ In addition, we can readily check that $w_{k+1}$ solves
\begin{equation} \label{Eq2.1}
	\hat{\Phi}_{{\varrho}}({D w_{k}}, x){\hat{F}}_{\varrho}(D^2 w_{k}, x)+{\hat{h}_{\varrho}}(x)\Abs{Dw_{k}}^{m}= {\hat{f}_{\varrho}}(x) \quad \text{in} \quad  B_{1}
\end{equation}
in the viscosity sense, where
\begin{align*}
	\hat{F}_{\varrho}(X,x):=&\nu^{-k\mu}F_{\varrho}\left(\nu^{k\mu}X,\nu^{k}x\right),\\
	{\hat{\Phi}}_{\varrho}(\xi,x):=&\nu^{-k(1+\mu)\hat{p}_{\varrho}(x)}\Phi_{\varrho}\left(\nu^{k(1+\mu)}\xi,\nu^{k}x\right) \;{\rm satisfies}\;\eqref{3atuihuaxing}\;{\rm with}\\
	\hat{p}_{\varrho}(x)&:={p}_{\varrho}(\nu^{k}x),\quad \hat{q}_{\varrho}(x):={q}_{\varrho}(\nu^{k}x), \quad
	\hat{a}_{\varrho}(x):=\nu^{k(1+\mu)(\hat{q}_{\varrho}(x)-\hat{p}_{\varrho}(x))}{a}_{\varrho}(\nu^{k}x),\\
	\hat{h}_{\varrho}(x):=&\nu^{-k\mu(1+\hat{p}_{\varrho}(x)-m)-k(\hat{p}_{\varrho}(x)-m)}h_{\varrho}(\nu^{k}x),\quad
	\hat{f}_{\varrho}(x):=\nu^{-k\mu(1+\hat{p}_{\varrho}(x))-k\hat{p}_{\varrho}(x)}f_{\varrho}(\nu^{k}x).
\end{align*}		
A direct calculation shows that
\begin{align*}
	{\rm osc}_{{\hat{F}_{\varrho}}}(x,0)
	&=\sup\limits_{M\in S^{d}\setminus \{0\}}\frac{\Abs{
			\nu^{-k\mu}F_{\varrho}\left(\nu^{k\mu}M,\nu^{k}x\right)-\nu^{-k\mu}F_{\varrho}\left(\nu^{k\mu}M,0\right)}}{\|M\|}\\
	&=\sup\limits_{M\in S^{d}\setminus \{0\}}\frac{\Abs{
			F\left(\nu^{k\mu}\varrho^{-2}M,\varrho\nu^{k}x\right)-F\left(\nu^{k\mu}\varrho^{-2}M,0\right)}}{\nu^{k\mu}\varrho^{-2}\|M\|}={\rm osc}_{{F}}(\varrho\nu^{k}x,0),
\end{align*}
then it follows that
\begin{equation*}
	\|{\rm osc}_{{\hat{F}_{\varrho}}}\|_{L^{\infty}\left(B_{1}\right)}=\|{\rm osc}_{{F}}\|_{L^{\infty}\left(B_{\varrho\nu^{k}}\right)}\leq C_{F}\varrho^{\theta}.
\end{equation*}
Next, we estimate the bound of $\hat{h}_{{\varrho}}$ and $\hat{f}_{{\varrho}}$. Applying $\mu=\min\left\{\frac{\gamma-q_{max}}{1+q_{max}},\frac{m+\beta-q_{max}}{1+q_{max}-m}\right\}$, in combination with \eqref{A4} and \eqref{A5}, we deduce that
\begin{align*}
	|\hat{h}_{{\varrho}}(x)|&\leq 
	K_{0}\nu^{-k\mu(1+q_{max}-m)-k(q_{max}-m-\beta)}\varrho^{2+p_{min}-m+\beta}
	\leq K_{0}\varrho^{2+p_{min}-m+\beta},\\
	|\hat{f}_{\varrho}(x)&\leq K_{1}\nu^{-k\mu(1+q_{max})-k(q_{max}-\gamma)}\varrho^{2+p_{min}+\gamma}\leq K_{1}\varrho^{2+p_{min}+\gamma}.
\end{align*}
At this point, in view of \eqref{canshuxuanqu}, $w_{k}$ falls into the framework of Lemma \ref{3aalem4.1}, and hence it yields
\begin{equation}\label{dierhanshu}
	\sup\limits_{x\in B_{\nu}}\Abs{w_{k}(x)}\leq \nu^{2+\mu}.	
\end{equation}
Rescaling \eqref{dierhanshu} back to $w_{1}$ yields
\begin{equation*}
	\sup\limits_{x\in B_{\nu^{(k+1)}}}\Abs{w_{1}(x)}\leq \nu^{(k+1)(2+\mu)}.
\end{equation*}
The induction argument is complete.

Finally, for $0<r\leq \nu\varrho$, there exists $k\in\mathbb{N}$, such that $\nu^{(k+1)}<\frac{r}{\varrho}\leq \nu^{k}$. Applying \eqref{disanhanshu} to obtain
\begin{align*}
	\sup\limits_{x\in B_{r}}\Abs{u(x)}&= 	\sup\limits_{x\in B_{r/\varrho}}\Abs{w_{1}(x)}\leq \sup\limits_{x\in B_{\nu^{k}}}\Abs{w_{1}(x)}\\
	&\leq \nu^{k(2+\mu)}=\nu^{(k+1)(2+\mu)-2-\mu}\\
	&\leq \left(\frac{r}{\varrho}\right)^{2+\mu}\nu^{-2-\mu}\\
	&=\left(\nu\varrho\right)^{2+\mu}r^{2+\mu}:=Cr^{2+\mu},
\end{align*}
where $C=C(d,\lambda,\Lambda,C_{F},K_{0},K_{1},\gamma,\beta,\theta,p_{min},q_{max},m,\nu)$. Therefore, $u$ is $C^{2,\mu}$ differentiable at 0, with $|Du(0)|=|D^{2}u(0)|=0$. The proof is complete now.
\end{proof}
\section*{Acknowledgments}
This work is supported by the National Natural Science Foundation of China (NSFC Grant No.12571103), Natural Science Foundation of Tianjin (Grant No. 25JCQNJC01400), and Young Scientific and Technological Talents (Level Three) in Tianjin.
\section*{Data availability} Data sharing is not applicable to this article as obviously no datasets were generated or analyzed during the current study.
\section*{Conflict of interest} Author states no conflict of interest.

\end{document}